\documentclass[12pt,reqno]{amsart}
\usepackage{amssymb}
\usepackage{txfonts}
\usepackage{amsmath}

 \newtheorem{thm}{Theorem}[section]

 \theoremstyle{definition}
 
 \theoremstyle{remark}
 
 \numberwithin{equation}{section}

\begin{document}

\title[]
{On the weighted forward reduced Entropy of Ricci
flow}

\author{Liang Cheng, Anqiang Zhu}

\dedicatory{}
\date{}

 \subjclass[2000]{
Primary 53C44; Secondary 53C42, 57M50.}

\keywords{ Ricci flow, weighted forward reduced volume, Type III
singularities, gradient expanding soliton}
\thanks{}

\address{Liang Cheng, School of Mathematics and Statistics, Huazhong Normal University,
Wuhan, 430079, P.R. CHINA}

\email{math.chengliang@gmail.com}

\address{Anqiang Zhu, School of Mathematics and Statistics, Wuhan University,
Wuhan, 430072, P.R. CHINA}

\email{anqiangzhu@yahoo.com.cn} \maketitle

\begin{abstract}
In this paper, we introduce the weighted forward reduced
volume of Ricci flow. The
weighted forward reduced volume, which related to expanders of Ricci flow, is well-defined on noncompact
manifolds and monotone non-increasing under Ricci flow.  Moreover, we show that, just the same as the
Perelman's reduced volume, the weighted reduced volume entropy has
the value $(4\pi)^{\frac{n}{2}}$ if and only if the Ricci flow is
the trivial flow on flat Euclidean space.
\end{abstract}

\section{Introduction}
In \cite{P1}, G.Perelman introduced the reduced entropy (i.e. reduced distance and reduced volume), which becomes one of powerful tools for studying Ricci flow.
The reduced entropy enjoys very nice analytic and geometric properties,
including in particular the monotonicity of the reduced volume. These
properties can be used, as demonstrated by Perelman, to show the limit of the suitable rescaled Ricci flows is a gradient shrinking soliton.

Then M.Feldman, T.Ilmanen, L.Ni \cite{FIN} observed that there is a dual version of G.Perelman's reduced entropy, which related to the expanders
of Ricci flow.
Let $g(t)$ solves the Ricci flow
\begin{align}\label{Ricci_flow}
    \frac{\partial g}{\partial t}=-2Rc.
\end{align}
 on $M\times
[0,T]$. Fix $x\in M^n$ and let $\gamma$ be a path $(x(\eta),\eta)$
joining $(x,0)$ and $(y,t)$. They
define the forward $\mathcal{L}_+$-length as
\begin{align}
    \mathcal{L}_+(\gamma)=\int^t_0\sqrt{\eta}(R(\gamma(\eta))+|\gamma'(\eta)|^2)d\eta.
\end{align}
Denote $L_+(y,t)$ be the length of a shortest forward
$\mathcal{L}_+$-length joining $(x,0)$ and $(y,t)$. Let
\begin{align}\label{reduced_l}
l_+(y,t)=\frac{L_+(y,t)}{2\sqrt{t}}
\end{align}
be the forward $l_+$-length. Note that the forward reduced distance
(\ref{reduced_l}) is defined under the forward Ricci flow
(\ref{Ricci_flow}), which is the only difference from Perelman's
reduced distance defined under the backward Ricci flow. The forward
reduced volume is defined in \cite{FIN} as
\begin{align}\label{compact_r_volume}
   \theta_+(t)=\int_{M} (t)^{-\frac{n}{2}}e^{l_+(y,t)}dvol(y).
\end{align}
 They also proved forward reduced volume
defined in (\ref{compact_r_volume}) is monotone non-increasing along
the Ricci flow (\ref{Ricci_flow}).

 Unfortunately, the forward
reduced volume defined in (\ref{compact_r_volume}) may not
well-defined on noncompact manifolds. In the first part of this
paper, we introduce the weighted forward reduced volume in
this paper based on the work in \cite{FIN} and \cite{P1}. The
weighted forward reduced volume is well-defined on noncompact
manifolds and monotone non-increasing under the Ricci flow
(\ref{Ricci_flow}). Moreover, we show that, just the same as the
Perelman's reduced volume, the weighted reduced volume entropy has
the value $(4\pi)^{\frac{n}{2}}$ if and only if the Ricci flow is
the trivial flow on flat Euclidean space.

We define the weighted forward reduced volume as follows. First, we
define the forward $\mathcal{L}_+$-exponential map $\mathcal{L}_+
exp(V,t) : T_xM \to M$ at time $t\in [0, T)$. For $V\in T_xM$, let
$\gamma_V$ denote the $\mathcal{L}_+$-geodesic such that
$\gamma_V(0)=p$, $\lim\limits_{t\to 0} \sqrt{t} \gamma'_V(t) = V$.
If $\gamma_V$ exists on $[0,t]$, we set
\begin{align}\label{Lexp}
\mathcal{L}_+ exp(V,t) = \gamma_V(t).
\end{align}
Denote $\tau_V$ be the first time the $\mathcal{L}_+$-geodesic
$\gamma_V$ stop minimizing. Define
\begin{align*}
    \Omega(t)=\{V\in T_xM^n:\tau_V>t\}.
\end{align*}
Obviously, $\Omega(t_2)\subset \Omega(t_1)$ if $t_1<t_2$. Let
$J_i^V(t), i=1,\cdots,n,$ be $\mathcal{L}_+$-Jacobi fields along
$\gamma_V(t)$ with
\begin{align}\label{Jacobi}
J_i^V(0)=0, (\nabla_V J_i^V)(0)=E^0_i,
\end{align}
where
$\{E^0_i\}^n_{i=1}$ is an orthonormal basis for $T_xM$ with respect
to $g(0)$. Then $D(\mathcal{L}_+ exp(V,t))(E^0_i)=J^V_i(t)$.
 We
define
\begin{align}\label{LJacobi}
\mathcal{L_+}J_V(t)=\sqrt{det(<J^V_i(t),J^V_j(t)>)}
\end{align}
and
the weighted forward reduced volume as
\begin{align}\label{WFRV}
\mathcal {\widetilde{V}}_+(t)=\int_{\Omega(t)} t^{-\frac{n}{2}}
e^{l_+(\gamma_V(t),t)}
e^{-2|V|^2_{g(0)}}\mathcal{L_+}J_V(t)dx_{g(0)}(V),
\end{align}
where $dx_{g(0)}$ is the standard Euclidean volume form on
$(T_xM,g(x,0))$, i.e. we define the weighted forward reduced volume as
\begin{align}\label{WFRV2}
\mathcal {\widetilde{V}}_+(t)=\int_{\Omega(t)} t^{-\frac{n}{2}}
e^{l_+(y,t)}e^{-2|\mathcal{L}_+ exp^{-1}(y,t)|^2_{g(0)}}dvol(y),
\end{align}
We use the convention
\begin{align*}
    \mathcal{L_+}J_V(t)\doteq 0 \text{ for } t\geq \tau_V.
\end{align*}
 Then we can write the weighted forward
reduced volume as
\begin{align}\label{WFRV1}
\mathcal {\widetilde{V}}_+(t)=\int_{T_xM^n} t^{-\frac{n}{2}}
e^{l_+(\gamma_V(t),t)}
\mathcal{L_+}J_V(t)e^{-2|V|^2_{g(0)}}dx_{g(0)}(V).
\end{align}

We remark that the density of forward reduced volume (\ref{compact_r_volume}) (i.e. $(t)^{-\frac{n}{2}}e^{l_+(y,t)}dvol(y)$)
is not pointwise monotone non-increasing under the Ricci flow (\ref{Ricci_flow}).
So it is not easy for us to add the weighted term to the forward reduced volume such that it could be
 defined on noncompact manifolds. In order to overcome this problem, we employ the G.Perelman's technique \cite{P1} that we pull the
 density of forward reduced volume back to the tangent space with the $\mathcal{L_+}$-exponential map (here we define the $\mathcal{L}_+$-exponential map
 as the similar way to \cite{P1}). Then we prove that the forward reduced volume density
 \begin{align*}
d\mathcal {V}_+=t^{-\frac{n}{2}} e^{l_+(\gamma_V(t),t)}
\mathcal{L_+}J_V(t)dx_{g(0)}(V),
\end{align*}
is pointwise monotone non-increasing under the Ricci flow (\ref{Ricci_flow}) with respect to $V\in T_xM$ (see Theorem \ref{R_Volume}). Just notice that (see the proof of Theorem \ref{WFRV_monotone})
\begin{align*}
    \lim\limits_{t\to 0^+} t^{-\frac{n}{2}} e^{l_+(\gamma_V(t),t)} \mathcal{L_+}J_V(t)=2^n e^{|V|^2_{g(0)}}.
\end{align*}
So we only need add the weighted term $e^{-2|V|^2_{g(0)}}$ to the density of forward reduced volume,
 which garrantees that the weighted forward reduced volume we defined in (\ref{WFRV}) is well defined on noncompact manifolds at $t=0$. Moreover, the
the weighted forward reduced volume is monotone non-increasing under the Ricci flow (\ref{Ricci_flow}) since we have
$$
\frac{d}{dt}(t^{-\frac{n}{2}} e^{l_+(\gamma_V(t),t)} \mathcal{L_+}J_V(t)e^{-2|V|^2_{g(0)}})=e^{-2|V|^2_{g(0)} }\frac{d}{dt}(t^{-\frac{n}{2}} e^{l_+(\gamma_V(t),t)} \mathcal{L_+}J_V(t))\leq 0
$$
for $V\in \Omega(t)$.

We exactly have the following properties for the weighted forward reduced
volume.
\begin{thm}\label{WFRV_monotone}
The weighted forward reduced volume defined in (\ref{WFRV}) is
monotone non-increasing under the Ricci flow (\ref{Ricci_flow}) and well-defined on complete noncompact manifolds. Moreover,
$\mathcal {\widetilde{V}}_+(t)\leq \lim\limits_{t\to 0+}\mathcal
{\widetilde{V}}_+(t)\leq(4\pi)^{\frac{n}{2}}$ for $t>0$. If $\mathcal
{\widetilde{V}}_+(t_1)=\mathcal {\widetilde{V}}_+(t_2)$ for some
$0<t_1<t_2$, then this flow is a gradient expanding soliton on
$0\leq t<\infty$ and hence is the trivial flow on flat Euclidean
space. In particular, if
 $\mathcal {\widetilde{V}}_+(\bar{t})=(4\pi)^{\frac{n}{2}}$ for some
time $\bar{t}>0$, then this flow is the trivial flow on flat
Euclidean space.
\end{thm}

We also have the following rescaling property for the weighted forward reduced volume.
\begin{thm}\label{WFRV_invariant}
We have $\mathcal {\widetilde{V}}^{j}_+(t)=\mathcal {\widetilde{V}}_+(\lambda_j^{-1}t)$ under the rescaling
$g_j(t)=\lambda_j g(\lambda_j^{-1}t)$, where $\mathcal {\widetilde{V}}^{j}_+$ and $\mathcal {\widetilde{V}}_+$ denote
the weighted forward reduced volume with respect to metric $g_j$ and $g$ respectively.
\end{thm}

The organization of the paper is as follows. In section 2, we first
recall some basic formulas and properties about forward reduced
entropy in \cite{FIN}. Then we study the properties of forward
reduced volume density which defined by forward
$\mathcal{L}_+$-exponential map. Finally, we give the proofs of
Theorem \ref{WFRV_monotone} and Theorem \ref{WFRV_invariant}.

\section{Weighted Forward Reduced volume and Expanders}

Before we present the proofs of Theorem \ref{WFRV_monotone} and
Theorem \ref{WFRV_invariant}, we recall some basic formulas and
properties about forward reduced entropy in \cite{FIN}. Clearly, one
can show that the $l_+$-length $l_+(y,t)$ is locally lipschitz function and the
cut-Locus of $\mathcal{L}_+ exp(V,t)$ is a closed set of measure
zero by using the similar methods in \cite{Y1}.

We need the following two theorems due to M.Feldman, T.Ilmanen, L.Ni
\cite{FIN}, which state the following adapted form.

\begin{thm}\cite{FIN}\label{variation}
Let $\gamma$ be a path $(x(\eta),\eta)$ joining $(x,0)$ and $(y,t)$.
Denote $L_+\doteq L_+(y,t)$ be the forward
$\mathcal{L}_+$-length joining $(x,0)$ and $(y,t)$.
Set $X=\gamma'(t)$ and $Y$ be a variational vector along $\gamma$
such that $Y(0)=0$. The first variation of $\mathcal{L}_+$ is that
\begin{align}\label{first_variation}
    \delta \mathcal{L}_+=2\sqrt{t}<X,Y>(t)+\int^t_0
    \sqrt{\eta}<Y,\nabla R-2 \nabla_X
    X+4Rc(X,\cdot)-\frac{1}{\eta}X>d\eta.
\end{align}
If $\gamma(t)$ is the mimimal $\mathcal{L}_+$-geodesic, then
\begin{align}\label{eq_nabla}
    \nabla L_+=2\sqrt{t}X,
\end{align}
\begin{align}\label{eq_KR}
    t^{\frac{3}{2}}(R+|X|^2)=K+\frac{1}{2}L_+,
\end{align}
where $K=\int^t_0 \eta^{\frac{3}{2}}H(X)d\eta$, $H(X)=\frac{\partial
R}{\partial t}+2<\nabla R,X>+2Rc(X,X)+\frac{R}{t}$. The second
variation of $\mathcal{L}_+$ is that
\begin{align}\label{eq_sec_variation}
    \delta^2_{Y} \mathcal{L}_+=& 2\sqrt{t}<X,Y>(t)+\int^t_0
    \sqrt{\eta}(HessR(Y,Y)-2R(X,Y,X,Y)\nonumber\\
    &+2|\nabla_X Y|^2+4\nabla_Y Rc(Y,X)-2\nabla_X Rc(Y,Y)) d\eta.
\end{align}
Let $\widetilde{Y}$ be a vector field along $\gamma$ satisfies the
ODE
\begin{equation} \label{eq_soliton_type}
\left\{
\begin{array}{ll}
\nabla_X\widetilde{Y}(\eta)=Rc(\widetilde{Y}(\eta),\cdot)+\frac{1}{2\eta}\widetilde{Y}(\eta),\eta\in [0,t]\\
\widetilde{Y}(0)=Y(0)=0.
\end{array}
\right.
\end{equation}
Then
\begin{align}\label{eq_Hess}
   Hess L_+(\widetilde{Y},\widetilde{Y})\leq
   \frac{|\widetilde{Y}|^2}{\sqrt{t}}+2\sqrt{t}Rc(\widetilde{Y},\widetilde{Y})-\int^t_0\sqrt{\eta}H(X,\widetilde{Y})d\eta,
\end{align}
where $H(X,\widetilde{Y})=-Hess
R(X,\widetilde{Y})+2R(X,\widetilde{Y},X,\widetilde{Y})+2|Rc(X,\cdot)|^2+\frac{Rc(\widetilde{Y},\widetilde{Y})}{t}+2\frac{\partial
Rc}{\partial t}(\widetilde{Y},\widetilde{Y})-4\nabla_{\widetilde{Y}}
Rc(\widetilde{Y},X)+4\nabla_X Rc(\widetilde{Y},\widetilde{Y})$. The
equality holds in (\ref{eq_Hess}) if and only if the vector filed
$\widetilde{Y}$ satisfying (\ref{eq_soliton_type}) is an
$\mathcal{L}_+$-Jacobi field.
\end{thm}

\begin{thm}\cite{FIN}\label{variation2}
Let $l_+\doteq l_+(y,t))=\frac{L_+(y,t)}{2\sqrt{t}}$ be the $l_+$-length
from $(x,0)$ to $(y,t)$. If $(y,t)$ is not in the cut-Locus of
$\mathcal{L}_+ exp$, then at $(y,t)$
\begin{align}
    \frac{\partial l_+}{\partial t}=R-\frac{l_+}{t}-\frac{K}{2t^{
    \frac{3}{2}}},\label{eq_l_1}\\
    |\nabla l_+|^2=\frac{l_+}{t}-R+\frac{K}{t^{
    \frac{3}{2}}},\label{eq_l_2}\\
    \Delta l_+ \leq R+\frac{n}{2t}-\frac{K}{2t^{
    \frac{3}{2}}},\\
    \frac{\partial l_+}{\partial t}+\Delta l_+ +|\nabla
    l_+|^2-R-\frac{n}{2t}\leq 0,\label{eq_1_4}\\
    2\Delta l_+ +|\nabla l_+|^2-R-\frac{l_++n}{t}\leq 0,\label{eq_1_5}
\end{align}
\end{thm}

We first study properties of the forward reduced volume density
defined as
\begin{align}\label{reduced_VD}
d\mathcal {V}_+=t^{-\frac{n}{2}} e^{l_+(\gamma_V(t),t)}
\mathcal{L_+}J_V(t)dx_{g(0)}(V),
\end{align}
where $\mathcal{L_+}J_V(t)$ defined in (\ref{LJacobi}).

Note that the weighted forward reduced volume
\begin{align*}
\mathcal {\widetilde{V}}_+(t)=\int_{T_xM^n}
e^{-2|V|^2_{g(0)}}d\mathcal {V}_+.
\end{align*}
Analogous to \cite{P1}, we have the following theorem.
\begin{thm}\label{R_Volume}
The forward reduced volume density $d\mathcal {V}_+$ defined in
(\ref{reduced_VD}) is monotone non-increasing along the Ricci flow
(\ref{Ricci_flow}). Moreover, if $d\mathcal {V}_+(t_1)=d\mathcal
{V}_+(t_2)$ for some $0<t_1<t_2$, then this flow is a gradient
expanding soliton.
\end{thm}
\begin{proof}
Let $\gamma_V(t)$ be the minimal $\mathcal{L}_+$-geodesic defined in
(\ref{Lexp}) and $y=\gamma_V(t)$. We consider $(y,t)$ in the
cut-Locus of $\mathcal{L}_+ exp(V,t)$.  Recall that $\nabla
l_+(y,t)=\gamma_V'(t)=X(t)$. Then by (\ref{eq_KR}) and
(\ref{eq_l_1}), we get
\begin{align}
    \frac{\partial l_+(\gamma_V(t),t)}{\partial t}&=\frac{\partial l_+(y,t)}{\partial
    t}+\nabla l\cdot X\nonumber\\
    &=R-\frac{l_+(y,t)}{t}-\frac{K}{2t^{
    \frac{3}{2}}}+ |X|^2\nonumber\\
    &=\frac{1}{2}t^{-\frac{3}{2}}K.
\end{align}
For any fixed $t$, we choose an orthonormal basis $\{E_i(t)\}$ of
$T_{\gamma_V(t)}M$. We extend $E_i(\eta)$, $\eta\in [0,t]$ to an
$\mathcal{L}_+$-Jacobi field along $\gamma_V$ with $E_i(0)=0$. We
write $J_i^V(t)=\sum\limits_i^n A^j_i E_j(t)$ for same matrix
$(A^j_i)\in GL(n,\mathbb{R})$. Then $J_i^V(\eta)=\sum\limits_i^n
A^j_i E_j(\eta)$ for all $\eta\in [0,t]$.

Hence, by (\ref{eq_Hess}), we calculate at time $t$
\begin{eqnarray}
   \frac{d}{d \eta}|_{\eta=t} \ln \mathcal{L_+}J_V
    &=&\frac{d}{d\eta}|_{\eta=t}\ln \sqrt{det(<\sum\limits_{k=1}^n A_i^k E_k,\sum\limits_{l=1}^n A_i^l E_l>)}\nonumber\\
    &=&\frac{1}{2}\frac{d}{d\eta}|_{\eta=t}\sum\limits_{i}|E_{i}|^2\nonumber\\
    &=&\sum\limits_{i}(-Rc(E_{i},E_{i})+<\nabla_{E_{i}} X,E_{i}>)\nonumber\\
    &=&\sum\limits_{i}(-Rc(E_{i},E_{i})+\frac{1}{2\sqrt{t}}Hess L_+(E_{i},E_{i}))\label{lJV}\\
    &\leq&
    \sum\limits_{i}(\frac{1}{2t}-\frac{1}{2\sqrt{t}}\int^t_0\sqrt{\eta}H(X,\widetilde{E}_{i})d\eta)\label{eq_2.5_1}
\end{eqnarray}
where $\widetilde{E}_{i}(\eta)$ are the vector fields along
$\gamma_V$ satisfying
\begin{equation} \label{eq_soliton_type}
\left\{
\begin{array}{ll}
\nabla_X\widetilde{E}_i(\eta)=Rc(\widetilde{E}_i(\eta),\cdot)+\frac{1}{2\eta}\widetilde{E}_i(\eta),\eta\in [0,t]\\
\widetilde{E}_i(t)=E_i(t),
\end{array}
\right.
\end{equation}
which in particular implies that
\begin{align}\label{eq_2.5_3}
    <\widetilde{E}_i,\widetilde{E}_j>(\eta)=\frac{\eta}{t}<E_i,E_j>(t)=\frac{\eta}{t}\delta_{ij}.
\end{align}
It follows that
\begin{align*}
    H(X,\widetilde{E}_{i})(\eta)=\frac{\eta}{t}K.
\end{align*}
Hence
\begin{align*}
     \frac{d}{d \eta}|_{\eta=t} \ln \mathcal{L_+}J_V \leq  \frac{n}{2t}-\frac{1}{2}t^{-\frac{3}{2}}K,
\end{align*}
and
\begin{align}\label{eq_2.5_2}
    \frac{d }{d t}\ln d\mathcal{V}_+ &=-\frac{n}{2t}+\frac{\partial l+}{\partial t}+\frac{d ln \mathcal{L_+}J_V}{d
    t}\leq 0.
\end{align}
If equality in (\ref{eq_2.5_2}) holds, then we have equality in
(\ref{eq_2.5_1}) holds. By Theorem \ref{variation}, we conclude that
each $\widetilde{E}_i(\eta)$ is an $\mathcal{L}_+$-Jacobi field.
Hence
\begin{align}\label{eq_2.5_4}
    \frac{d}{d\eta}|_{\eta=t}|E_i|^2=\frac{d}{d\eta}|_{\eta=t}|\widetilde{E}_i|^2=\frac{|E_i(t)|^2}{t}.
\end{align}
Combining with (\ref{lJV}) and (\ref{eq_2.5_4}), we get
\begin{align*}
    Rc(E_i,E_i)-\frac{1}{2\sqrt{t}}Hess
    L_+(E_i,E_i)=-\frac{|E_i|^2}{2t}.
\end{align*}
\end{proof}

Now we can give the proof of Theorem \ref{WFRV_monotone}.

 \textbf{Proof of Theorem \ref{WFRV_monotone}.}
By Theorem \ref{R_Volume}, we know that
$$
\frac{d}{dt}(t^{-\frac{n}{2}} e^{l_+(\gamma_V(t),t)} \mathcal{L_+}J_V(t))\leq 0
$$
for $V\in \Omega(t)$.
It follows that
$$
\frac{d}{dt}(t^{-\frac{n}{2}} e^{l_+(\gamma_V(t),t)} \mathcal{L_+}J_V(t)e^{-2|V|^2_{g(0)}})\leq 0
$$
for $V\in \Omega(t)$.
Moreover, since we have
$\Omega(t_2)\subset \Omega(t_1)$,
$$\mathcal {\widetilde{V}}_+(t_2)\leq \mathcal
{\widetilde{V}}_+(t_1)$$
for $t_1<t_2$ .  We calculate that
\begin{align*}
    \lim\limits_{t\to 0+}l_+(\gamma_V(t),t)&=\lim\limits_{t\to 0+} \frac{1}{2\sqrt{t}}\int^t_0\sqrt{\eta}(R(\gamma_V(\eta),\eta)+|\frac{d \gamma_V}{d\eta}|^2)d\eta\\
    &=\lim\limits_{t\to 0+}t(R(\gamma_V(t),t)+|\frac{d \gamma_V}{dt}|^2) \\
    &=|V|^2_{g(0)}.
\end{align*}
Let
$J_i^V(t), i=1,\cdots,n,$ be $\mathcal{L}_+$-Jacobi fields along
$\gamma_V(t)$ with
\begin{align}\label{Jacobi}
J_i^V(0)=0, (\nabla_V J_i^V)(0)=E^0_i,
\end{align}
where
$\{E^0_i\}^n_{i=1}$ is an orthonormal basis for $T_xM$ with respect
to $g(0)$. Since $(\nabla_V J_i^V)(0)=E^0_i$ and $V=\lim\limits_{t\to 0} \sqrt{t} \gamma'_V(t) $, we get
\begin{align*}
    \lim\limits_{t\to 0^+}\frac{\mathcal{L_+}J_V(t)}{t^{\frac{n}{2}}}=\lim\limits_{t\to 0^+}\frac{\sqrt{det(<2\sqrt{t}E_i(t),2\sqrt{t}E_j(t)>_{g(0)})}}{t^{\frac{n}{2}}}=2^n,
\end{align*}
we conclude that
\begin{align*}
   \lim\limits_{t\to 0^+} t^{-\frac{n}{2}} e^{l_+(\gamma_V(t),t)} \mathcal{L_+}J_V(t)=2^n e^{|V|^2_{g(0)}}.
\end{align*}
Hence
\begin{align*}
    \lim\limits_{t\to 0^+}\mathcal {\widetilde{V}}_+(t)\leq \int_{T_pM}2^n
    e^{-|V|^2_{g(0)}}dx(V)=(4\pi)^{\frac{n}{2}}.
\end{align*}
If $\mathcal {\widetilde{V}}_+(t_1)=\mathcal {\widetilde{V}}_+(t_2)$
for any $0<t_1<t_2$, then $d\mathcal {V}_+(t_1)=d\mathcal
{V}_+(t_2)$ for any $0<t_1<t_2$. So $(M^n,g(t))$ must be a gradient
expanding soliton by Theorem \ref{R_Volume}, i.e. we have
\begin{align*}
    Rc+Hess (-l_+)=-\frac{g}{2t}
\end{align*}
for some smooth function $l_+$ on $M^n$. Let $\phi_t:M\to M, 0<t\leq
\bar{t}$ be the one-parameter family of diffeomorphisms obtained by
\begin{align*}
    \frac{d\phi_t}{dt}=\nabla l_+\quad\text{and}\quad \phi_{\bar{t}}=Id.
\end{align*}
We consider $h(t)=\frac{\bar{t}}{t}\phi_t^*g(t)$ and calculate
\begin{align*}
    \frac{d h}{d t}
    &=-\frac{\bar{t}}{t^2}\phi_t^*g(t)+\frac{\bar{t}}{t}\phi^*_t\mathcal{L}_{\frac{d\phi_t}{dt}}(g(t))-2\frac{\bar{t}}{t}\phi^*_tRc(g(t))\\
    &=-\frac{\bar{t}}{t^2}\phi_t^*g(t)+\frac{\bar{t}}{t}2Hess (l_+)+\frac{\bar{t}}{t}\phi^*_t(\frac{g}{t}-2Hess(l_+))=0.
\end{align*}
It follows that
\begin{align*}
    g(t)=\frac{t}{\bar{t}}(\phi^{-1}_t)^*g(\bar{t}).
\end{align*}
Suppose that there is some $(y,\bar{t})$ with $|Rm|(y,\bar{t})=K>0$,
we have $|Rm|(\phi^{-1}_t(y),t)=\frac{K\bar{t}}{t}$, and these
curvatures are not bounded as $t\to 0$, which is a contradiction.
Then we have
\begin{align*}
    Hess(l_+)=\frac{1}{2t}g.
\end{align*}
Thus $l_+$ is strictly convex function. The similar arguments to
Lemma 2.3 in \cite{Y1} can show that
$$
l_+(y,t)\geq e^{-2ct}\frac{d_{g(0)}(x,y)}{4t}-\frac{nc}{3}t,
$$
if $Rc\geq -cg$ on $[0,t]$, so that $l_+(y,t)$ have the only minimum
point in $M^n$. Hence $M^n$ is diffeomorphic to $\mathbb{R}^n$.

 Since $\mathcal {\widetilde{V}}_+(t)$ is monotone
non-increasing, $\mathcal {\widetilde{V}}_+(t)$ is independent of
$t$ if $\mathcal {\widetilde{V}}_+(\bar{t})=(4\pi)^{\frac{n}{2}}$
for some time $\bar{t}>0$. Then we derive that $M^n$ is isometric to
$\mathbb{R}^n$.
  $\Box$

Finally, we give the proof of Theorem \ref{WFRV_invariant}.

 \textbf{Proof of Theorem \ref{WFRV_invariant}.}We denote $\gamma_v^j(t)$  be the
minimal $\mathcal{L}_+$-geodesic with respect to $g_j(t)$ which starting from $(x_j,0)$  and satisfying $\lim\limits_{t\to 0} \sqrt{t}
\frac{d\gamma^j_V(t)}{dt} = V$.

We have that
$\gamma^{j}_{\sqrt{\lambda_j^{-1}}V}(t)=\gamma_V(\lambda_j^{-1}t)$,
$l_+^{j}(y,t)=l_+(y,\lambda_j^{-1}t)$ and
$\mathcal{L_+}J^j_{V}(t)dx_{g_j(0)}(V)=(\lambda_j^{-1})^{-\frac{n}{2}}\mathcal{L_+}J_{\sqrt{\lambda_j}V}(\lambda_j^{-1}t)dx_{g(0)}(\sqrt{\lambda_j}V)$.
Hence
\begin{align*}
\mathcal {\widetilde{V}}^{j}_+(t)
&=\int_{T_xM^n}(t)^{-\frac{n}{2}}e^{l^j_+(\gamma_{V}(t),t)}\mathcal{L_+}J^j_{V}(t)e^{-2|V|^2_{g_j(0)}}dx_{g_j(0)}(V)\\
&=\int_{T_xM^n}(\lambda_j^{-1}t)^{-\frac{n}{2}}
e^{l_+(\gamma_{\sqrt{\lambda_j}V}(\lambda_j^{-1}t),\lambda_j^{-1}t)}\\
&\qquad\qquad\quad\times\mathcal{L_+}J_{\sqrt{\lambda_j}V}(\lambda_j^{-1}t)e^{-2|\sqrt{\lambda_j}V|^2_{g(0)}}dx_{g(0)}(\sqrt{\lambda_j}V)\\
&= \mathcal {\widetilde{V}}_+(\lambda_j^{-1}t).
\end{align*}
$\Box$

\thanks{\textbf{Acknowledgement}: We would like to express our gratefulness to our thesis advisor professor
Li Ma for his constant support and encouragement.}

\end{document}